\documentclass[a4paper, 11pt,psamsfonts]{amsart}
\pdfoutput=1
\usepackage{amsmath}
\usepackage{amsthm}
\usepackage{amssymb}
\usepackage{amscd}
\usepackage{amsfonts}
\usepackage{amsbsy}
\usepackage{enumerate}
\usepackage{graphicx}
\usepackage{calc}
\usepackage{xcolor}
\usepackage{setspace}
\usepackage{graphicx}
\usepackage[dvips]{psfrag}
\usepackage{array}
\usepackage{color}
\usepackage{epsfig}
\usepackage{url}

\newtheorem {theorem} {Theorem}
\newtheorem {proposition}{Proposition}
\newtheorem {corollary}{Corollary}
\newtheorem {lemma}{Lemma}

\newtheorem {remark}{Remark}
\newtheorem {example} {Example}

\title[Jacobian conjecture in $\mathbb R^2$]
{Jacobian conjecture in $\mathbb R^2$}

\author[
X. Zhang]
{
Xiang Zhang }

\address{School of Mathematical Sciences and MOE--LSC, Shanghai Jiao Tong University, Shanghai 200240, People's Republic of China} \email{xzhang@sjtu.edu.cn}



\begin{document}

\begin{abstract}
Jacobian conjecture states that if $F:\ \mathbb C^n(\mathbb R^n)\rightarrow \mathbb C^n(\mathbb R^n)$ is a polynomial map such that the Jacobian of $F$ is a nonzero constant, then $F$ is injective.

This conjecture is still open for all $n\ge 2$, and for both $\mathbb C^n$ and $\mathbb R^n$.
Here we provide a positive answer to the Jacobian conjecture in $\mathbb R^2$ via the tools from the {theory of} dynamical systems.
\end{abstract}

\maketitle

{
\bf \noindent Remark on this revised version:

\small After the version arXiv:2011.12701 was published, I received emails from several specialists, who wanted to see the detail proofs on some arguments, which were originally only a sentence statement for each one. This version contains all detail proofs on those arguments, and it is different from arXiv:2011.12701 in the following aspects:
\begin{itemize}
\item Add Lemma 1 and its detail proof in pages 6 to 9. It was originally only a sentence statement in page 5.

\item Add Proposition 1 and its proof in pages 14 and 15, which is used in the proof of Lemma 1.

\item Add a detail proof on the argument that $\Omega_*:=\bigcup\limits_{t=0}\limits^{\beta_{p_*}}\phi_t(\Upsilon_*)$ covers one of the half spaces limited by $\Upsilon_*$ in page 16, which was originally only a sentence statement on page 12.

\item Modify some print and grama mistakes, and revise a few presentations.
\end{itemize}
}

Let $(f,g):\ \mathbb R^2(\mathbb C^2)\rightarrow \mathbb R^2(\mathbb C^2)$  be a polynomial map. We denote by $J(f,g)$ the Jacobian matrix of the map $(f,g)$, and by $D(f,g)$ the Jacobian of $(f,g)$, i.e. $D(f,g)=\det J(f,g)$. In what follows $\mathbb R^*=\mathbb R\setminus\{0\}$ and $\mathbb C^*=\mathbb C\setminus\{0\}$.

The classical Jacobian conjecture states that if the Jacobian $D(f,g)=1$, then $F$ is injective. This conjecture was first posed as a question by Keller \cite{Ke39} in 1939 { on polynomial maps with integer coefficients, and had been taken the nowadays name and widely publicized by Shreeram Abhyankar}. For more information on the history of this conjecture, see e.g. the survey paper{s and books} \cite{BCW82,Es00,Wr81}. Nowadays,  the Jacobian conjecture is formulated in the next form (see e.g. Bass et al \cite{BCW82} and { van den Essen} \cite[pages XV and 82]{Es00}
).
\smallskip

\noindent{\bf Jacobian conjecture.}
{\it
If $F: \mathbb R^n(\mathbb C^n)\rightarrow \mathbb R^n(\mathbb C^n)$ is a polynomial map such that the Jacobian $DF\in\mathbb R^*(\mathbb C^*)$, then $F$ is injective.
}
\smallskip


Associated to the Jacobian conjecture, Randall \cite{Ra83} in 1983 posed the so called real Jacobian conjecture.
\smallskip

\noindent{\bf Real Jacobian conjecture.}
{\it
If $F: \mathbb R^2\rightarrow \mathbb R^2$ is a polynomial map such that the Jacobian $DF$ of $F$ does not vanish, then $F$ is injective.
}
\smallskip

This conjecture is not correct in general, as illustrated by Pinchuck \cite{Pi94} in 1994, which we will recall it again later on.

A general formulation in $\mathbb C^n$ of the real Jacobian conjecture was posed { by Smale \cite{Sm98} in 1998 as his 16th problem on a list of 18 open mathematical problems}. For distinguishing it from the Jacobian conjecture mentioned above and according to Pinchuck \cite{Pi94}, we call it Strong Jacobian conjecture.
\smallskip

\noindent{\bf Strong Jacobian conjecture.} {\it
If $F: \mathbb C^n \rightarrow \mathbb C^n $ is a polynomial map such that the Jacobian $DF$ of $F$ does not vanish, then $F$ is injective.
}
\smallskip

{We note that in the strong Jacobian conjecture in $\mathbb C^n$ but not in $\mathbb R^n$, the Jacobian $DF$ must be a nonzero constant.}

For the polynomial maps, there is a natural connection between injective and surjective. Bia{\l}ynicki-Birula and   Rosenlicht \cite{BR62} in 1962 proved the next result (see also Rudin \cite{Ru95} in 1995).
\smallskip

\noindent{\bf Bia{\l}ynicki-Birula and Rosenlicht Theorem.} {\it If $F:\ \mathbb K^n\rightarrow \mathbb K^n$ is an injective polynomial map, with $\mathbb K$ an algebraically closed field of characteristic zero, then $F$ is a polynomial automorphism from $\mathbb K^n$ to itself.}
\smallskip

Pinchuck \cite{Pi94} in 1994 constructed a counterexample to the real Jacobian conjecture, which is a polynomial map in $\mathbb R^2$ with nonvanishing Jacobian, and it is not injective. 
As mentioned by the author, one of the components of the constructed map has its level curves containing more than one branches.

According to the Pinchuck's example, in order that the real Jacobian conjecture holds, there needs some additional conditions. Along this direction, there appeared some additional sufficient conditions ensuring the real Jacobian conjecture holds. See for instance
\cite{BD10,BGL16,BL15,BO16,CGM96,FGR04,GG21, IL19,Sa93,Sa98}, parts of which were proceeded via tools from the theory of dynamical systems on the characterization of global centers for real planar polynomial vector fields. But, at the moment there is no a necessary and sufficient condition for which the real Jacobian conjecture holds.

{On both Jacobian and strong Jacobian conjectures, there are many partial results, the relations of Jacobian conjecture with other conjectures, and the relations with other subjects}, see e.g. the survey papers by Bass et al \cite{BCW82} in 1982 and {van den Essen} \cite{Es00} in 2000, and the recent publications \cite{ACH19, BBRY14, Ca14, Ca141,DEZ17, DP18, EL15, FMV14, Fr14, GST16, JZ17, LBP17, Mi19, Pa14, PS18, St16, St19, St18, VGG17,YB13,YB14}. On the history of the Jacobian and strong Jacobian conjectures we should also mention the papers by Abhyankar and Moh \cite{AM75} in 1975 for two dimensional case, by
Wang \cite{Wa80} in 1980 for the map $F$ with degrees no more than $2$, by Bass et al \cite{BCW82} in 1982, Yagzhev \cite{Ya80} in 1980 and Dru\.{z}kowski \cite{Dr83} in 1983 via reduction of degrees of the polynomial maps, and by Hubbers \cite{Hu94} on cubic maps in dimension $4$.

In any case, as our knowledge, the Jacobian and strong Jacobian conjectures are in general open for all $n\ge 2$, see {for instance, van den Essen} \cite[Section 10.2]{Es00} in 2000, {Sabatini \cite[lines 1-3 of page 160]{Sa01} in 2001}, the recent paper \cite[paragraph 1 of page 5251]{BGL16} in 2016, and also the item Jacobian conjecture in Wikipedia.

Here we provide a positive answer to the Jacobian conjecture in $\mathbb R^2$, and its proof will be down by using the tools from the { theory of} dynamical systems related with the commuting Hamiltonian vector fields and the Poincar\'e compactification.

\begin{theorem}\label{tj}
If $F: \mathbb R^2\rightarrow \mathbb R^2$ is a polynomial map such that the Jacobian $DF\in\mathbb R^*$, then $F$ is injective.
\end{theorem}

\begin{proof}
Set $F(x,y)=(f(x,y),g(x,y))\in (\mathbb R[x,y])^2$. Then $DF=D(f,g):=f_xg_y-f_yg_x$.
Let $H_f=(-f_y,f_x)^T$ and $H_g=(-g_y,g_x)^T$ be respectively the Hamiltonian vector fields associated to $f$ and $g$, where $T$ represents the transpose of a matrix. Denote by $[H_f,H_g]$ the Lie bracket of the two Hamiltonian vector fields $H_f$ and $H_g$. Recall that the Lie bracket of two smooth vector fields $\mathcal X$ and $\mathcal Y$ is by definition $[\mathcal X, \ \mathcal Y]=(J\mathcal Y)\mathcal X-(J\mathcal X)\mathcal Y$, where $J\mathcal X$ and $J\mathcal Y$ are the Jacobian matrices of the vector fields $\mathcal X$ and $\mathcal Y$, respectively. For more information on Lie brackets, see e.g. \cite{AM78,Ol86}.

{The next seven claims will complete the proof of Theorem \ref{tj}.}

\noindent{\bf Claim }1. {\it $D(f,g)\in\mathbb R$ if and only if $[H_f,H_g]\equiv 0$.}

This statement can be found, for example, in \cite[Lemma 2]{Sa01} and \cite[Remark 3.2]{CL19}. For completeness, we present its proof here.

\begin{proof}[Proof of Claim 1]
Direct calculations show that
\begin{align*}
D_x&=-f_yg_{x^2}+f_xg_{yx}+f_{x^2}g_y-f_{yx}g_x,\\
D_y&=-f_yg_{xy}+f_xg_{y^2}+f_{xy}g_y-f_{y^2}g_x,
\end{align*}
where $D_x$ represents the partial derivative of $D(f,g)$ with respect to $x$. In addition, some computations verify that
\[
[H_f,H_g]=(JH_g)H_f-(JH_F)H_g=(-D_y,D_x)^T=H_D.
\]
If follows that $[H_f,H_g]\equiv 0$ if and only if $D_x\equiv 0$ and $D_y\equiv 0$, and if and only if $D(f,g)$ is a constant. The claim follows. 
\end{proof}

We remark that for a given polynomial or an analytic vector field $\mathcal X$, the set of vector fields $\mathcal Y$ satisfying $[\mathcal X,\ \mathcal Y]=0$ is called {\it centralizer} of $\mathcal X$. Recently Cerveau and Lins Neto \cite{CL19} and Walcher and Zhang \cite{KWZ20} characterized the structure and dimension of the centralizer.

\noindent{\bf Claim }2. {\it If $D(f,g)\in\mathbb R^*$, then the Hamiltonian vector fields $H_f$ and $H_g$ both have no singularities and have no periodic orbits. Consequently, all orbits of $H_f$ and $H_g$ have their positive and negative limits going to infinity.}

\begin{proof}[Proof of Claim 2] Indeed, if the Hamiltonian vector field $H_f$ has a singularity at $P_0$, then $(-f_y(P_0),\ f_x(P_0))=0$. So the gradient vector field $\nabla f:=(f_x,\ f_y)$ vanishes at $P_0$. It forces that $D(f,g)(P_0)=0$, a contradiction {with the assumption of the claim}.

If the polynomial Hamiltonian vector field $H_f$ has a periodic orbit, saying $\Gamma_0$, then it will have a period annulus containing $\Gamma_0$ in its interior. { By Ye \cite[Theorem 1.6]{Ye86} the outer and inner boundaries of a maximal period annulus of an analytic vector fields both contain singularities (those on the outer boundary may be at infinity).
Or using the well known fact that in the region limited by a periodic orbit there is at least one singularity of the planar vector field, we will be in a contradiction with the last argument that $H_f$ cannot have singularities in the finite plane.}

As a consequence of these last two facts, all orbits of $H_f$ are regular and have their positive and negative limits both at infinity. The claim holds. \end{proof}


\smallskip

\noindent{\bf Claim }3. {\it If $D(f,g)\in\mathbb R^*$, then the Hamiltonian vector fields $H_f$ and $H_g$ intersect transversally everywhere. Moreover, the flow of $H_g$ maps any orbit of $H_f$ to a different one and preserves the time.}

\begin{proof}[Proof of Claim 3] Indeed, the first argument follows from the fact that
\[
\det(H_f, H_g)
=\det J(f,g)=D(f,g)\ne 0.
\]
Claim 1 {has shown} that $[H_f,H_g]\equiv 0$. Then the second argument is a well known fact in the world of dynamical systems, and is from the {commutation} of the flows of two commuting vector fields $H_f$ and $H_g$ and their transversality everywhere. For more information on dynamical properties of commuting vector fields, see e.g. {\cite[Theorem 1.34]{Ol86}}, or \cite{Sa97}, or \cite{AM78}. The claim follows. \end{proof}

\noindent{\bf Claim }4. {\it For the polynomial $f\in\mathbb R[x,y]$, its associated Hamiltonian vector field $H_f$ has finitely many singularities at the infinity, i.e. on the equator of the Poincar\'e disc.}

Poincar\'e compactification is useful in the study of global dynamics, including the infinity, of polynomial vector fields in $\mathbb R^n$. For more information on Poincar\'e compactification, see e.g. \cite{CL90}, \cite[Chapter 5]{DLA06}.

\begin{proof}[Proof of Claim 4] It is well known \cite[Chapter 5]{DLA06} that the infinity of a polynomial vector field of degree $m$ either contains at most $m+1$ singularities on the equator of the Poincar\'e disc or fulfils singularities on the equator of the Poincar\'e disc, where each pair of diametral points on the equator is numerated as one point.
Let a polynomial vector field of degree $m$ be of the form $\mathcal X_P:=\left(\sum\limits_{j=0}\limits^m p_j(x,y),\sum\limits_{j=0}\limits^m q_j(x,y)\right)^T$ with $p_j,\ q_j$ being homogeneous polynomials of degree $j$ or naught, and with $p_m$ and $q_m$ not both identically zero.
By \cite{CL91, SZ18, Zh11,Zh03} in order that the infinity of $\mathcal X_P$ on the Poincar\'e disc fulfils singularities, it is necessary that $(p_m(x,y),q_m(x,y))$ is of the form $(xh_{m-1}(x,y),yh_{m-1}(x,y))$, where $h_{m-1}$ is a homogeneous polynomial of degree $m-1$. But it is not possible for the polynomial Hamiltonian vector field $H_f$. In fact, suppose that $f$ is of degree $\ell$, and { its highest order homogenous part}  $f_\ell(x,y)=\sum\limits_{j=0}\limits^{\ell}\alpha_jx^jy^{\ell-j}$, $\alpha_j\in\mathbb R$. Then
the highest order term of the Hamiltonian vector field $H_f$ is $H_{f_\ell}:=(-\partial_yf_\ell,\ \partial_xf_\ell)=\left(-\sum\limits_{j=0}\limits^{\ell}\alpha_j(\ell-j)x^jy^{\ell-j-1},\ \sum\limits_{j=0}\limits^{\ell}\alpha_jjx^{j-1}y^{\ell-j}\right)$, where $\partial_yf_\ell$ is the partial derivative of $f_\ell$ with respect to $y$. In order for its first (resp. second) component of $H_{f_\ell}$ to have a factor $x$ (resp. $y$), one has $\alpha_0=\alpha_\ell=0$.  So  $H_{f_\ell}=\left(x(-\sum\limits_{j=1}\limits^{\ell-1}(\ell-j)\alpha_jx^{j-1}y^{\ell-j-1}), \ y(\sum\limits_{j=1}\limits^{\ell-1} j \alpha_jx^{j-1}y^{\ell-j-1})\right)$. Clearly, $H_{f_\ell}$ is not of the form $(xh_{\ell-2},yh_{\ell-2})$ with $h_{\ell-2}$ being homogeneous of degree $\ell-2$. This proves that the highest order term of the Hamiltonian vector field $H_f$ does not have the form such that its infinity fulfils singularities. Consequently, the claim holds. \end{proof}

{We remark that in the following Proposition \ref{r5}, we will further prove that the singularities at the infinity are uniquely determined by the real linear factors of the highest order homogeneous parts of $f$.}

In order to prove the {injectivity} of the polynomial map $F=(f,g)$ in $\mathbb R^2$, it is equivalent to prove $F(p_0)\ne F(q_0)$ for all $p_0,q_0\in\mathbb R^2$ and $p_0\ne q_0$, i.e. $f(p_0)\ne f(q_0)$ or $g(p_0)\ne g(q_0)$.

For $p,q\in\mathbb R^2$ satisfying $p\ne q$, if $f(p)\ne f(q)$, we are down.
Assume that there exist some $p_0,q_0\in\mathbb R^2$ with $p_0\ne q_0$, such that $f(p_0)=f(q_0)$. Set $u_0:=f(p_0)=f(q_0)$ and $\mathcal S:=\{r\in\mathbb R^2|\ f(r)=u_0\}$.

Let $\phi_t$ and $\psi_s$ be the flows of the Hamiltonian vector fields $H_f$ and $H_g$, respectively.

\noindent{\bf Claim }5. {\it  If $D(f,g)\in\mathbb R^*$, then any connected branch of $\mathcal S$ sweeps the full plane $\mathbb R^2$ under the action of the flow $\psi_s$ of $H_g$.}


\begin{proof}[Proof of Claim 5] Let $\mathcal S_1$ be one of the connected branches of $\mathcal S$ and assume without loss of generality that $p_0$ is located on $\mathcal S_1$. Set $v_0=g(p_0)$ and
$\mathcal R=\{r\in\mathbb R^2|\ g(r)=v_0\}$. Let $\mathcal R_1$ be the connected branch of $\mathcal R$ with $p_0$ located on it.

Now we work on the Poincar\'e compactification for planar polynomial vector fields, see e.g. \cite{CL90},
\cite[Chapter 5]{DLA06}. Recall from Claim 2 that $\mathcal R$ and $\mathcal S$ are nonsingular curves and are not closed in the finite plane. Then $\mathcal R_1$ and $\mathcal S_1$ all have their {endpoints} at the infinity.
Thus, $\mathcal S_1$ separates the Poincar\'e disc in two disjoint regions, saying $D_{\mathcal S_1}^+$ and $D_{\mathcal S_1}^-$, and its two endpoints at the infinity are denoted by $I_{11}$ and $I_{12}$ (they maybe coincide, i.e. $I_{11}=I_{12}$), which are singularities of the Hamiltonian vector field $H_f$ at the infinity.

As it is well known, the flow of an analytic vector field on a compact space is complete, i.e. any of its solutions is defined on all $\mathbb R$ (see e.g. \cite{DLA06}). So the Poincar\'e compactifications of the Hamiltonian vector fields $H_f$ and $H_g$ have their flows complete. That is, the flows $\phi_t$ of $H_f$ and $\psi_s$ of $H_g$ are complete on the Poincar\'e disc. Since we are working in Poincar\'e compactification, in what follows all the proofs related to the flows will be processed in the full $\mathbb R$. {In fact, we will show in Proposition \ref{r4} that in our setting the flows of the Hamiltonian vector fields $H_f$ and $H_g$ are also complete in $\mathbb R^2$.}

By Claim 3 the vector fields $H_f$ and $H_g$ are transversal everywhere in the finite plane, and by Claim 1 one has $[H_f,H_g]\equiv 0$ and so the flows of $H_f$ and $H_g$ {commute}. It follows that under the action of the flow $\psi_s$ of $H_g$, the invariant curve $\mathcal S_1$ of $H_f$ moves towards either into $D_{\mathcal S_1}^+$ or into $D_{\mathcal S_1}^-$. For fixing notations and without loss of generality, we assume that $\psi_s(\mathcal S_1) \subset D_{\mathcal S_1}^+$ when $s>0$ and $\psi_s(\mathcal S_1) \subset D_{\mathcal S_1}^-$ when $s<0$. By { commutation} of the flows $\phi_t$ and $\psi_s$, it follows that for any fixed $s\in\mathbb R$, $\psi_s(\mathcal S_1)$ is an invariant curve of $H_f$, and for different values of $s$ the invariant curves $\psi_s(\mathcal S_1)$'s are different. Recall that invariant curves of a planar Hamiltonian vector field are all given by the level curves of its associated Hamiltonian function.
It forces that $\psi_s(\mathcal S_1)$ is a level curve of $f$, and so it is regular and has its two endpoints at the infinity.

Note that when the time $s$ continuously increases or decreases, the invariant curve $\psi_s(\mathcal S_1)$ of $H_f$ continuously runs forward from $\mathcal S_1$ inside $D_{\mathcal S_1}^+$  or moves backward from $\mathcal S_1$ inside $D_{\mathcal S_1}^-$.
{ On the endpoints of $\psi_s(\mathcal S_1)$ one has the next result.
\begin{lemma}\label{lep}
When $s$ continuously varies the two endpoints $($maybe coincide$)$ of $\psi_s(\mathcal S_1)$ on the Poincar\'e disc will keep the same as those of $\mathcal S_1$.
\end{lemma}
\begin{proof}
 { Recall from Claim 4 that the singularities at the infinity of $H_f$ on the Poincar\'e disc are finite, and the fact via commutation of $\phi_t$ and $\psi_s$ that $\psi_s(\mathcal S_1)$ is an orbit of the Hamiltonian vector field $H_f$ having its endpoints at the infinity for any fixed $s\in\mathbb R$ which belongs to the existence interval of the soluton. It follows that $f(\psi_s(\mathcal S_1))$ is a constant for each fixed $s$, and so $\psi_s(\mathcal S_1)$ is an invariant algebraic curve of $H_f$.
Some calculations (see for instance \eqref{er5}) show that for any $p\in \mathcal S_1$
\[
\frac{d(f(\psi_s(p)))}{ds}=-c, \quad s\in J_p,
\]
where $J_p$ is the existence interval of the solution passing $p$, and $c\ne 0$ is the constant such that $D(f(x,y),g(x,y))\equiv c$ on $\mathbb R^2$ by the assumption of the Jacobian conjecture. Then for any fixed $s\ne 0$ (maybe small) one has
\[
f(\psi_s(p))=f(p)-cs,
\]
and consequently
\begin{equation}\label{eif1}
f(\psi_s(\mathcal S_1))=f(\mathcal S_1)-cs,
\end{equation}
where we have used the fact that $\mathcal S_1=\psi_0(\mathcal S_1)\ni p$ and $\psi_s(\mathcal S_1)\ni \psi_s(p)$ are both orbits of the Hamiltonian vector field $H_f$.

If for any $s\ne 0$, the two orbits $\psi_s(\mathcal S_1)$ and $\mathcal S_1$ have the same endpoints at the infinity of the Poincar\'e disc, we are down.

If not, for any sufficiently small $s_0\ne 0$ (without loss of generality we consider $s_0>0$) there exists an $s_0^*$ with $s_0^*\in(0, s_0]$ such that $\psi_{s_0^*}(\mathcal S_1)$ and $\mathcal S_1$ have different endpoints. For simplifying notation, without loss of generality we can take $s_0^*=s_0$. See Figure \ref{F0} for an illustration, where $b_0$ is an any chosen and fixed point on $\mathcal S_1$. Of course, we can take without loss of generality that $I_{11}$ and $I_{21}$ are two consecutive singularities at the infinity on the equator of the Poincar\'e disc with $I_{11}$ being one of the endpoints of $\mathcal S_1$ and with $I_{21}$ being that of $\psi_{s_0}(\mathcal S_1)$. In the next proof we do not pay attention on where another endpoint of $\psi_{s_0}(\mathcal S_1)$ coincides with $I_{12}$ or not, where $I_{12}$ is another endpoint of $\mathcal S_1$ (maybe coincide with $I_{11}$).

\begin{figure}[ht]
\begin{center}
\includegraphics[width=7cm]{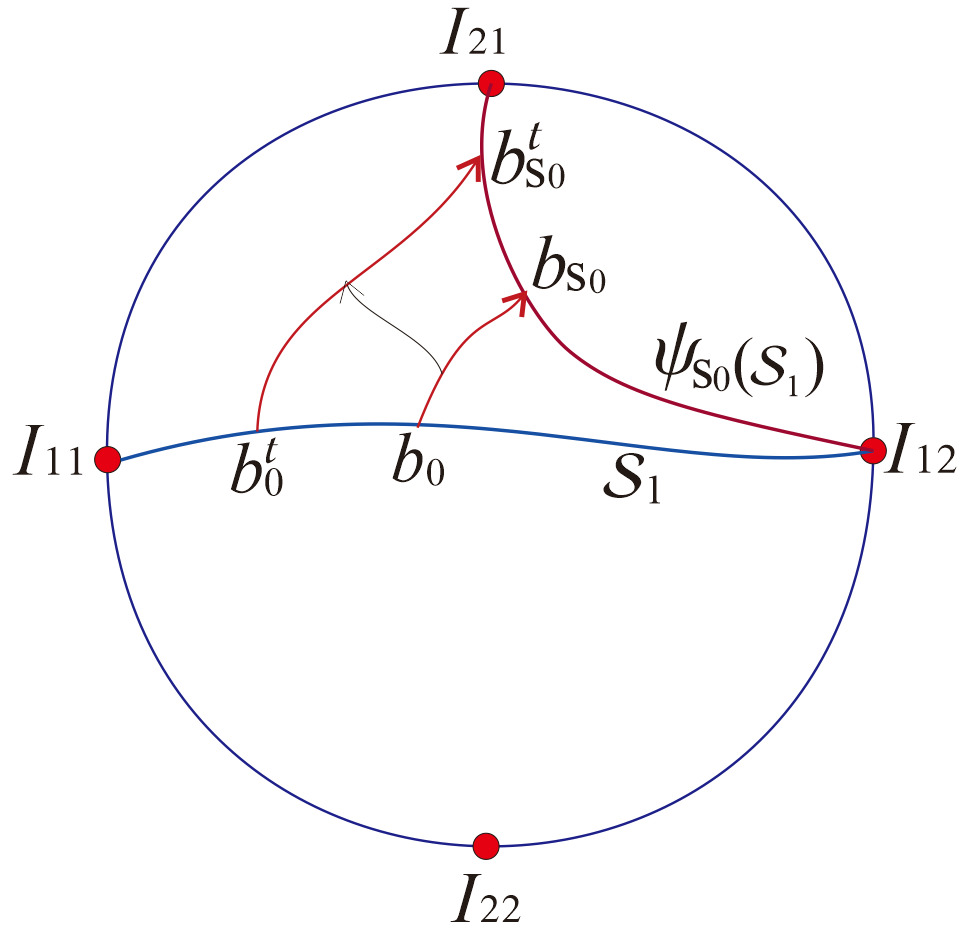}
\end{center}
\caption{Endpoints of $\mathcal S_1$ and $\psi_s(\mathcal S_1)$ at the infinity of the Poincar\'e disc by the contrary assumption} \label{F0}
\end{figure}

Now considering the flow $\phi_t$ acting on $b_0$, and assuming without loss of generality that when $t>0$, $\phi_t(b_0)$  moves towards $I_{11}$ along $\mathcal S_1$.
Denote by $b_0^t=\phi_t(b_0)$ and $b_{s_0}^t=\phi_t(b_{s_0})$. It follows from commutation of the flows $\phi_t$ and $\psi_s$ that $b_{s_0}^t=\phi_t(b_{s_0})$.
Again by commutation of $\phi_t$ and $\psi_s$ one has that the flow $\phi_t$ of $H_f$ pushes the orbit segment $\overrightarrow{b_0b_{s_0}}$ of $H_g$  forward to $\overrightarrow{b_0^tb_{s_0}^t}$, and it sweeps the full region $\Theta_{s_0}$ limited by the orbit segments $\overrightarrow{b_0I_{11}}$ and $\overrightarrow{b_{s_0}I_{21}}$ of $\phi_t$ and the orbit segment $\overrightarrow{b_0b_{s_0}}$ of $\psi_s$ when $t\rightarrow \infty$ (hereafter the over arrow denotes the orbit arc in the positive sense corresponding to the arrow).
Indeed, if not, there will be a point, saying $\xi_0$, in the interior of $\Theta_0$, and $\xi_0\in\lim\limits_{t\rightarrow\infty}\phi_t(\overrightarrow{b_0b_{s_0}})$.
Then by transversality of the Hamiltonian vector fields $H_f$ and $H_g$ in $\mathbb R^2$ and
 monotonicity of the orbit segment $\overrightarrow{b_0^tb_{s_0}^t}$ in $t$, it follows that the negative orbit of $\phi_t$ passing $\xi_0$ will be located in between the orbit segments $\overrightarrow{b_0I_{11}}$ and $\overrightarrow{b_{s_0}I_{21}}$ of $H_f$ and be back to $\overrightarrow{b_0b_{s_0}}$ at some point saying $\eta_0$.
This process can be realized due to commutation of the flows $\phi_t$ and $\psi_s$ and the extension theorem of solutions.
 Now the orbit segment $\overleftarrow{\xi_0\eta_0}$ is a part of a regular orbit of $H_f$ in the finite plane, and so the time from $\xi_0$ to $\eta_0$ under the negative action of the flow $\phi_t$ is finite. We are in contradiction with the choice of $\xi_0$. This verifies the argument that under the action of $\phi_t$ with $t\in[0,\infty)$, the orbit argument $\overrightarrow{b_0b_{s_0}}$ sweeps the full region $\Theta_{s_0}$.

From \eqref{eif1} and the above argument one gets that $f$ takes values in between $f(\mathcal S_1)$ and $f(\mathcal S_1)-cs_{s_0}$ in $\Theta_{s_0}$ with $s_0>0$ being arbitrarily small constant. But it is impossible because $f$ is a polynomial and $\Theta_{s_0}$ is an infinitely large half--open region. This impossibility can also be verified in the following way:
\begin{itemize}
\item Since $H_f$ has only finitely many singularities at the infinity, for simplifying notation and without loss of generality we can assume that the endpoints of the $y$--axis are not singularities of $H_f$ at the infinity (otherwise, we can take a suitable rotation to the coordinate system such that the hypothesis holds). Let $f$ be of degree $m$ and its highest order homogeneous part be $f_m(x,y)$. Set $f_m(x,y)=\ell(y-a_1x)^{k_1}\ldots (y-a_px)^{k_p}h_q(x,y)$ with $a_1<\ldots<a_p$, where $h_q$ is a homogeneous polynomial of degree $q$ which cannot be factorized in linear factors in $\mathbb R$, and $k_1+\ldots+k_p+q=m$,
where we have used the fact that there is no a singularity which is located at the endpoints of the $y$--axis. Then $a_1,\ldots,a_p$ uniquely determine the singularities of $H_f$ on the equator of the Poincar\'e disc in one to one correspondence.
For readers' convenience we provide a proof on this last argument in Proposition \ref{r5}.

\item Let $a_j$ and $a_{j+1}$ be respectively the slopes of $\mathcal S_1$ and $\psi_{s_0}(\mathcal S_1)$ at the infinity. Note from Proposition \ref{r5} that $a_j$ and $a_{j+1}$ are two different real zeros of $f_m(1,w)$.
    The previous proof verifies that in $\Theta_{s_0}$, $ {x^2+y^2}$ can take any sufficiently large value, and $y/x$ can take any value in between $a_j$ and $a_{j+1}$ for $ {x^2+y^2}$ arbitrarily large (in case that $a_j$ and $a_{j+1}$ have different sign, for instance $a_j<0$ and $a_{j+1}>0$, $y/x$ could undergo from $a_j$ to $-\infty$ and from $\infty$ to $a_{j+1}$). Dividing the two sides of \eqref{eif1} by $(\sqrt{x^2+y^2})^m$, we can write the equality \eqref{eif1} in the next form
    \begin{equation}\label{eif2}
    f_m\left(\frac{x}{\sqrt{x^2+y^2}},\frac{y}{\sqrt{x^2+y^2}}\right)=O\left(\frac{1}{\sqrt{x^2+y^2}}\right),
    \quad \begin{array}{c} (x,y)\in\Theta_{s_0},\\
      {x^2+y^2}\rightarrow \infty,
     \end{array}
    \end{equation}
    where the capital $O$ represents the same order infinitesimal quantity  as $({\sqrt{x^2+y^2}})^{-1}$ during the limit process.
The above proof also shows that there exists a passage, saying $\mathfrak J$, located in the interior of $\Theta_{s_0}$, which satisfies  $y/x\rightarrow k$ when $(x,y)\in\mathfrak J$ and $ x^2+y^2 \rightarrow\infty$, where $k$ is a constant belonging in between $a_j$ and $a_{j+1}$ but not equal to $a_j, \ a_{j+1}$. Taking the limit $x^2+y^2\rightarrow\infty$ along $\mathfrak J$ to the two sides of \eqref{eif2} yields
\begin{equation}\label{eif3}
\frac{(\pm  1)^m}{\sqrt{1+k^2}} f_m (1,k )=0,
\end{equation}
where $1^m$ and $(-1)^m$ depend on $x>0$ and $x<0$, respectively, in the limit process. But the equality \eqref{eif3} cannot hold because $f_m(1,k)\ne 0$. We are again in contradiction.
\end{itemize}

The above proofs ensure that the contrary assumption on $\mathcal S_1$ and $\psi_s(\mathcal S_1)$ having different endpoints at the infinity is absurd. As a result, $\psi_s(\mathcal S_1)$ for $s\ne 0$ must have the same slope at the infinity as that of $\mathcal S_1$, and so all $\psi_s(\mathcal S_1)$ have the same endpoints at the infinity on the Poincar\'e disc as those of $\mathcal S_1$.

It completes the proof of the lemma.
}
\end{proof}

We note that this lemma will play the key role in the remaining proof of this claim.

Now we consider the branch $\mathcal R_1$ of $\mathcal R$. As shown in the previous proof for $\mathcal S_1$, one has that }
the two endpoints (maybe coincide) of $\mathcal R_1$ are also at the infinity of $H_g$ on the Poincar\'e disc, we denote them by $J_{11}$ and $J_{12}$ (maybe $J_{11}=J_{12}$), which are singularities of the Hamiltonian vector field $H_g$ at the infinity. Then it follows from the transversality of the vector fields $H_f$ and $H_g$ that {the locations of $I_{11}$, $I_{12}$, $J_{11}$ and $J_{12}$ have four possibilities:
\begin{itemize}
 \item[$(a_1)$] $I_{11}\ne I_{12}$, $J_{11}\ne J_{12}$, and $\{J_{11}, J_{12}\}\cap\{I_{11}, I_{12}\}=\emptyset$. So $J_{11}$ and $J_{12}$ are located on the different sides of $\mathcal S_1$, saying $J_{11}\in D_{\mathcal S_1}^+$ and $J_{12}\in D_{\mathcal S_1}^-$. The simple map $(f,g)=(x, y)$ is in this situation.
 \item[$(a_2)$] $I_{11}\ne I_{12}$, $J_{11}\ne J_{12}$, and $\{J_{11}, J_{12}\}\cap\{I_{11}, I_{12}\}\ne\emptyset$. So $J_{11}$ and $J_{12}$ coincide with $I_{11}$ and $I_{12}$ at least at one of them, saying for example $J_{11}=I_{11}$.
     The map $(f,g)=(x, y-x^3)$ illustrates this case.
 \item[$(a_3)$] $I_{11}\ne I_{12}$ and $J_{11}=J_{12}$. Then one of $I_{11}$ and $I_{12}$ coincides with $J_{11}=J_{12}$. The map $(f,g)=(x, y-x^2)$ is in this case.
 \item[$(a_4)$] $I_{11}= I_{12}$ and $J_{11}\ne J_{12}$. This is also in the situation $(a_3)$.
 \item[$(a_5)$] $I_{11}= I_{12}$ and $J_{11}= J_{12}$. Then $I_{11}=I_{12}=J_{11}=J_{12}$. The map $(f,g)=(y-x^2, y+x-x^2)$ illustrates this case.
   \end{itemize}
Now we consider the actions of the flow $\psi_s$ on $\mathcal S_1$, and of the flow $\phi_t$ on $\mathcal R_1$. Fig.~\ref{Fd} illustrates these actions.

\begin{figure}[ht]
\begin{center}
\begin{tabular}{cc}
\psfrag{A}{$I_{11}$}\psfrag{B}{$I_{12}$}\psfrag{C}{$J_{11}$}\psfrag{D}{$J_{12}$}\psfrag{P}{$p_0$}
\psfrag{S}{$\mathcal S_{1}$}\psfrag{R}{{\color{purple}$\mathcal R_{1}$}}
\psfrag{E}{$\psi_{-s_2}(\mathcal S_{1})$}\psfrag{F}{$\psi_{-s_1}(\mathcal S_{1})$}
\psfrag{H}{$\psi_{s_1}(\mathcal S_{1})$}\psfrag{I}{$\psi_{s_2}(\mathcal S_{1})$}
\psfrag{U}{{\color{purple}$\phi_{t_1}(\mathcal R_{1})$}}\psfrag{V}{{\color{purple}$\phi_{t_2}(\mathcal R_{1})$}}
\psfrag{G}{$I_{11}=J_{11}$}\psfrag{J}{$I_{12}=J_{12}$}\psfrag{K}{$I_{11}=I_{12}=J_{12}$}\psfrag{L}{$I_{11}=I_{12}=J_{11}=J_{12}$}
\includegraphics[width=5cm]{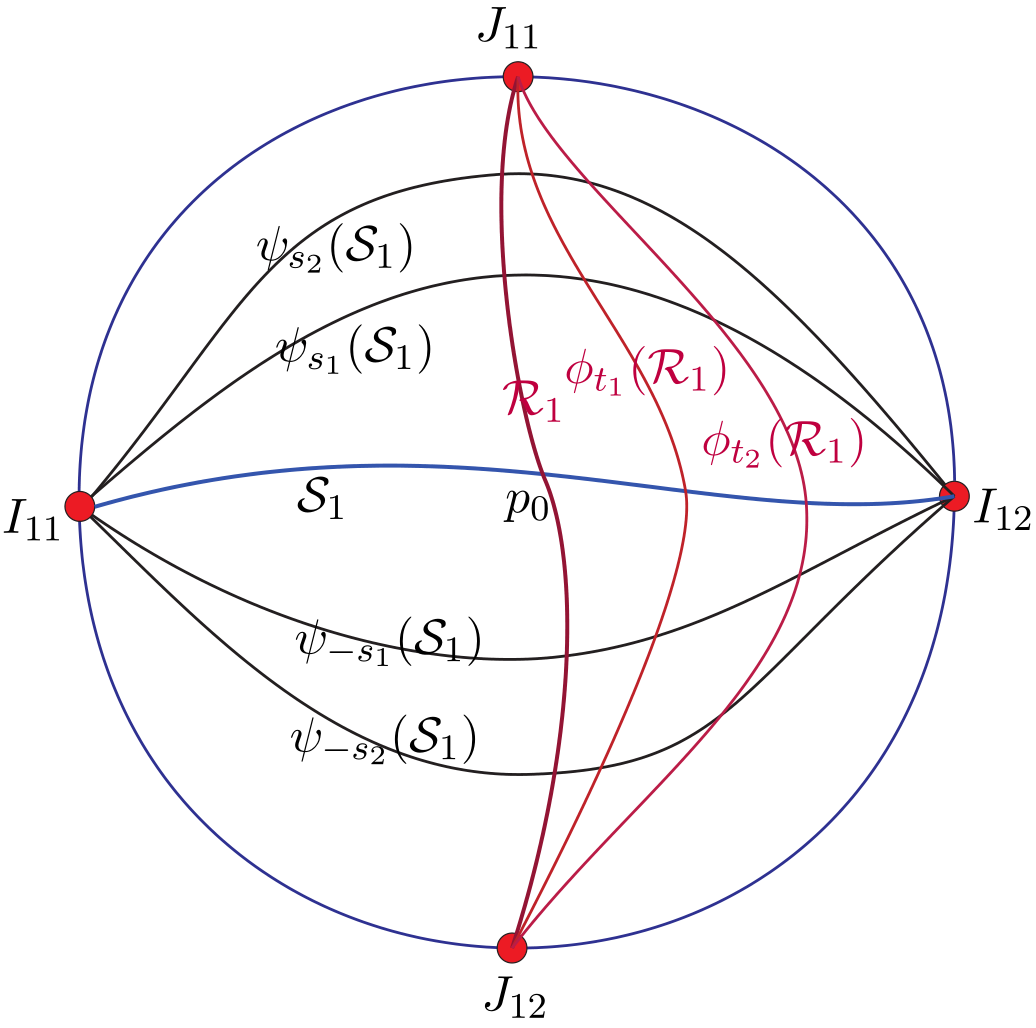} & \qquad \includegraphics[width=5cm]{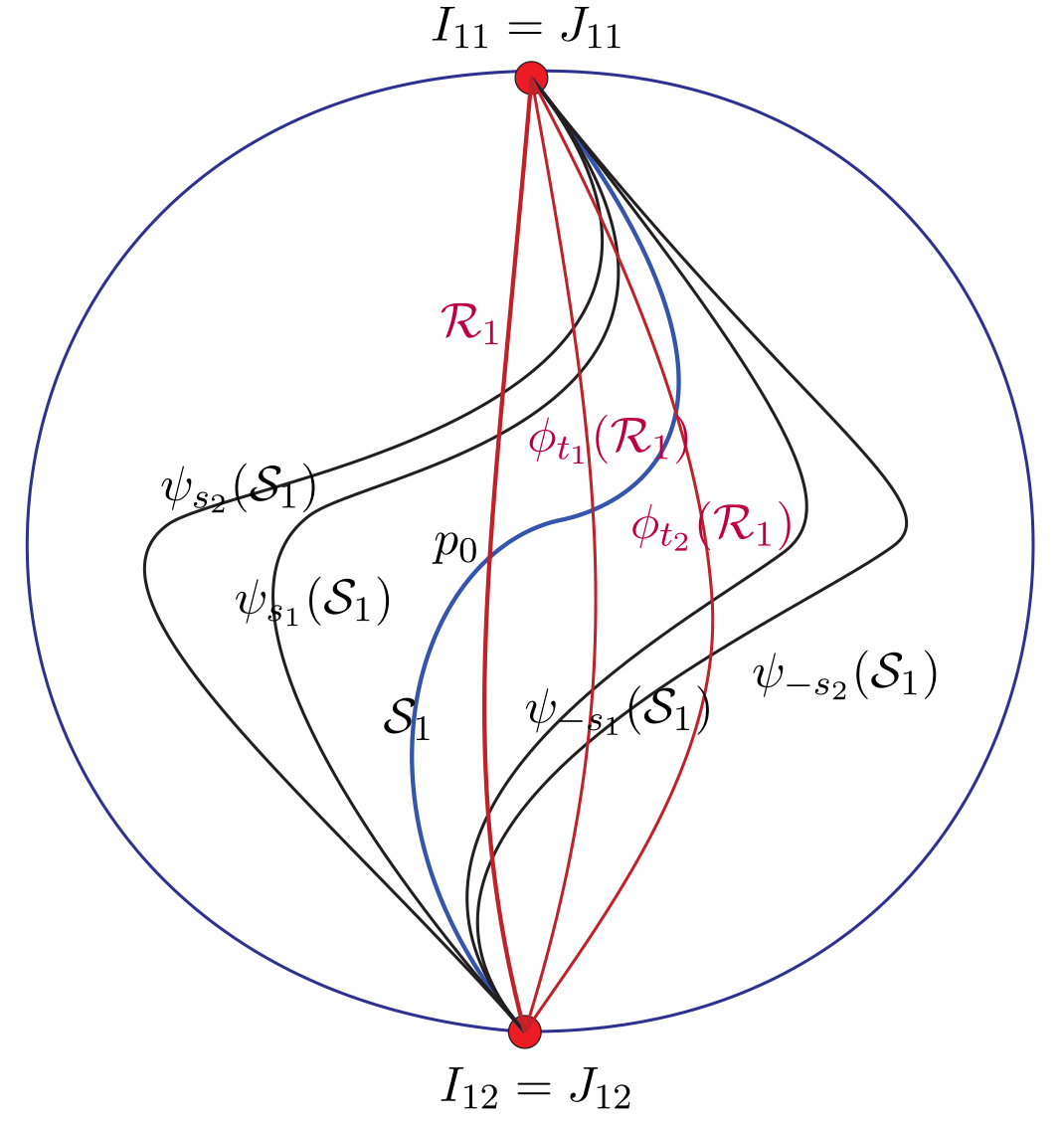} 
\\
$(a_1)$   & \  $(a_2)$ 
\\
\psfrag{A}{$I_{11}$}\psfrag{B}{$I_{12}$}\psfrag{C}{$J_{11}$}\psfrag{D}{$J_{12}$}\psfrag{P}{$p_0$}
\psfrag{S}{$\mathcal S_{1}$}\psfrag{R}{{\color{purple}$\mathcal R_{1}$}}
\psfrag{E}{$\psi_{-s_2}(\mathcal S_{1})$}\psfrag{F}{$\psi_{-s_1}(\mathcal S_{1})$}
\psfrag{H}{$\psi_{s_1}(\mathcal S_{1})$}\psfrag{I}{$\psi_{s_2}(\mathcal S_{1})$}
\psfrag{U}{{\color{purple}$\phi_{t_1}(\mathcal R_{1})$}}\psfrag{V}{{\color{purple}$\phi_{t_2}(\mathcal R_{1})$}}
\psfrag{G}{$I_{11}=J_{11}$}\psfrag{J}{$I_{12}=J_{12}$}\psfrag{K}{$I_{11}=I_{12}=J_{12}$}\psfrag{L}{$I_{11}=I_{12}=J_{11}=J_{12}$}
\includegraphics[width=5cm]{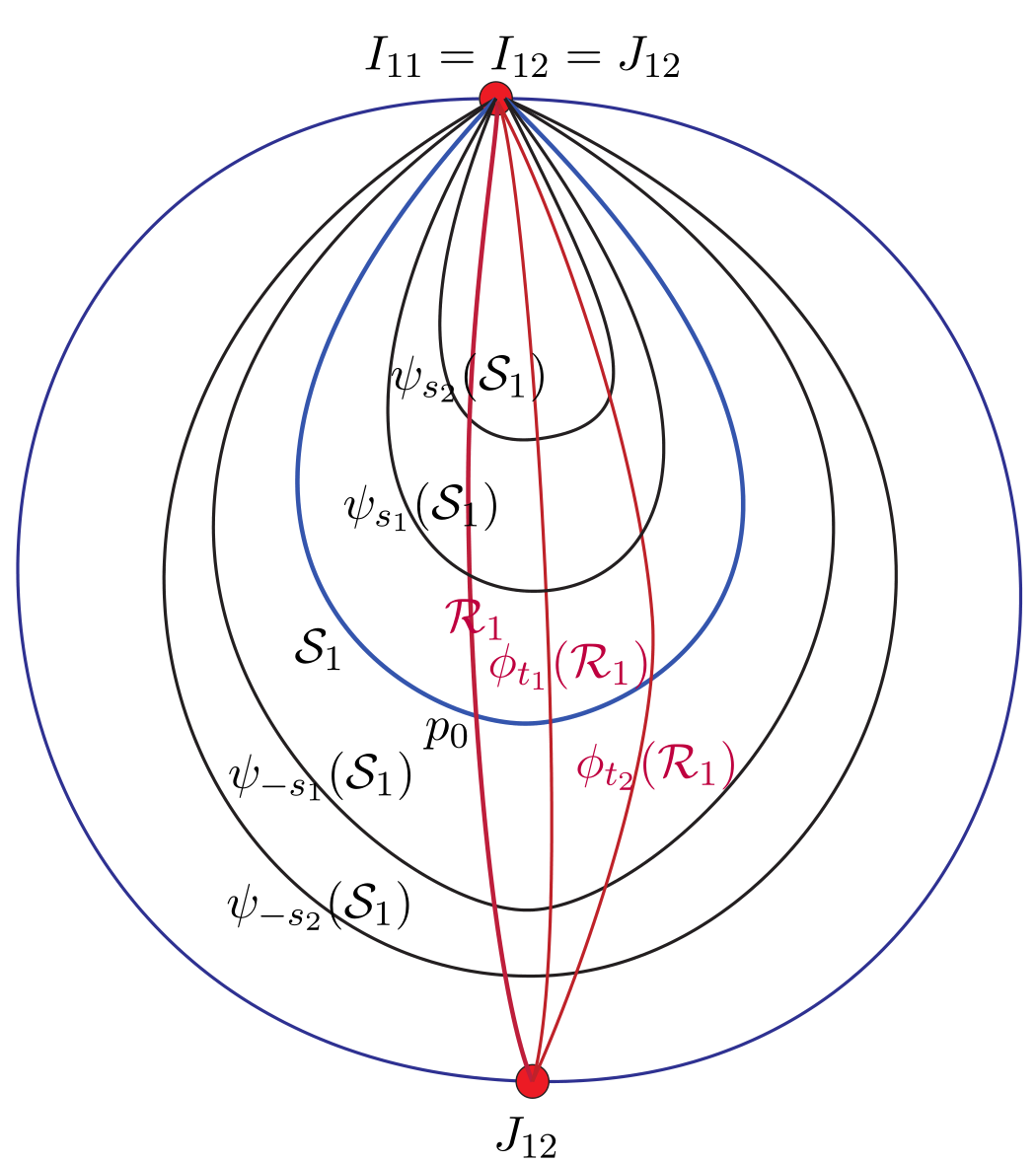}  &\qquad \includegraphics[width=5cm]{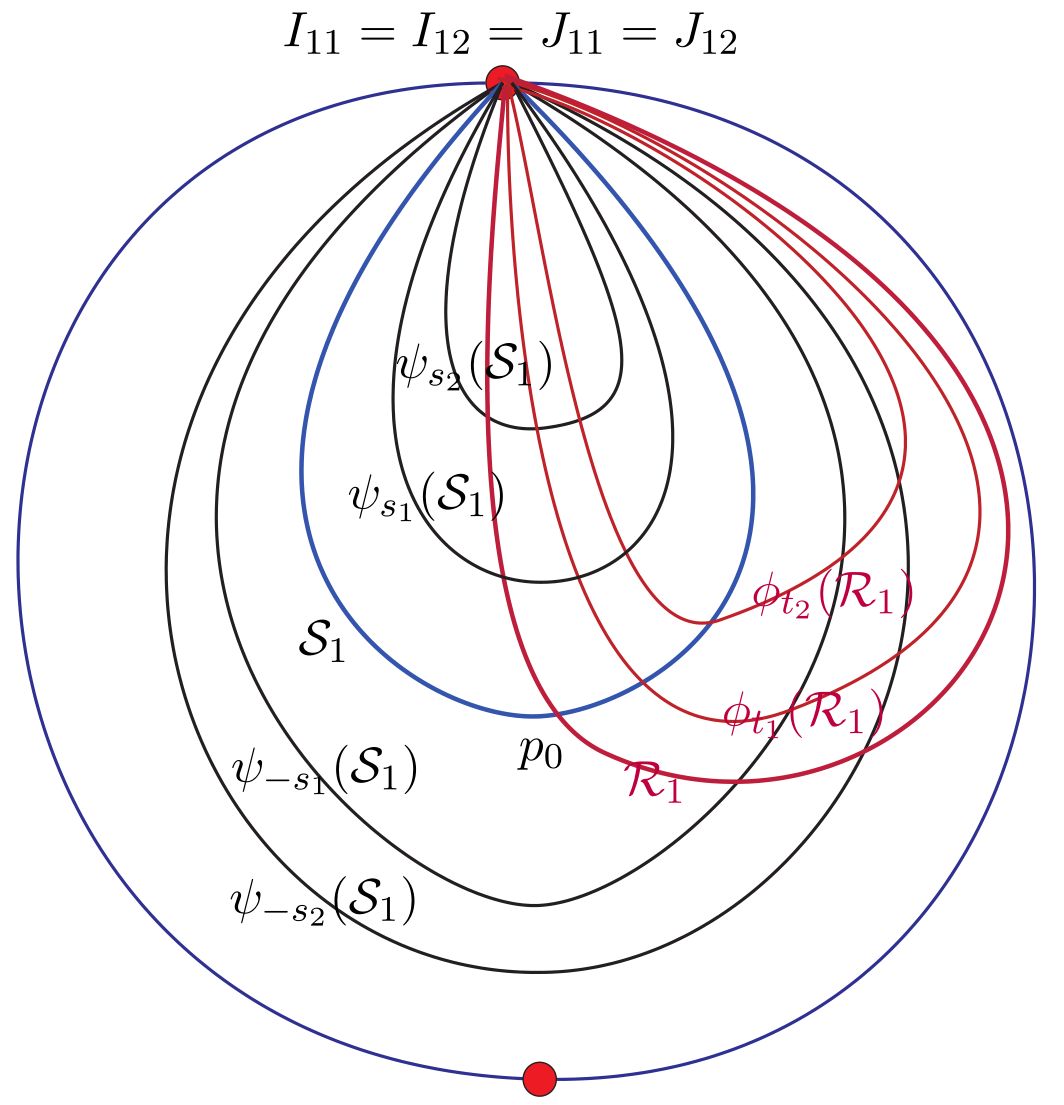}
\\
$(a_3)$  & \quad  $(a_4)$
\end{tabular}
\end{center}
\caption{Here we illustrate for $s_2>s_1>0$ and $t_2>t_1>0$} \label{Fd}
\end{figure}

}

When $s$ increases from $0$ to $\infty$, the flow $\psi_s$ acts on $p_0$ and pushes it along $\mathcal R_1$ to the infinity $J_{11}$, because $\mathcal R_1$ is invariant under the action of the flow $\psi_s$ of $H_g$. Denote by $\mathcal R_1^+$ the part of $\mathcal R_1$ from $p_0$ to $J_{11}$. Applying the similar arguments as presented in the previous paragraphs, one gets that for any fixed $t\in\mathbb R$, $\phi_t(\mathcal R_1^+)$ is positively invariant under the action of $\psi_s$, and it is a part of one level curve  of $g$. Furthermore, the same arguments as in the proof of Lemma \ref{lep} verify that for all $t\in\mathbb R$, $\phi_t(\mathcal R_1^+)$ have the same endpoint at $J_{11}$ at the infinity of $H_g$ on the Poincar\'e disc. This also implies that all $\phi_t(\mathcal R_1^+)$ {for $t\in\mathbb R$} approach the infinity of $\mathbb R^2$ in the same slope{, and so one of the endpoints of all $\phi_t(\mathcal R_1^+)$ has the same limit as a singularity of $H_g$ at the infinity}. As a consequence of these last facts {we can conclude that} $\bigcup\limits_{s=0}\limits^\infty\psi_s(\mathcal S_1)=D_{\mathcal S_1}^+$.

{The next is a proof to this last argument. By contrary, if not, set $\Omega_1:=\bigcup\limits_{s=0}\limits^\infty\psi_s(\mathcal S_1)$ and $\Omega_0:=D_{\mathcal S_1}^+\setminus \Omega_1$, then $\Omega_0$ will contain a nonempty open region and so is $\Omega_1\setminus\{\mathcal S_1\}$. Denote by $\partial\Omega_0$ the boundary of $\Omega_0$ which is located in $\mathbb R^2$ (excluding the part at the infinity).
It is well known that the boundary of $\Omega_0$ is formed by orbits of $H_f$.
Let $B_1^*$ be one of the regular orbits of $\partial\Omega_0$, and its endpoints be $I_1^*$ and $I_2^*$ (they may coincide), which are at the infinity. Hence $B_1^*$ must be a heteroclinic or a homoclinic orbit of $H_f$. Since the Hamiltonian vector fields $H_f$ and $H_g$ are transversal everywhere in the finite plane, it follows that under the action of the flow $\psi_s$,  $B_1^*$ will positively get into the interior of $\Omega_0$, and negatively moves into $\Omega_1$.
Now the flows $\psi_s$ and $\phi_t$ commute, it forces that the negative action of $B_1^*$ under the flow $\psi_s$ will be an orbit of $H_f$, choosing one of them, and denoting it by $O^*$. Then, as shown in the previous proof, $O^*$ will have its endpoints same as those of $B_1^*$. Since $B_1^*$ belongs to the boundary of $\Omega_1=\bigcup\limits_{s=0}\limits^\infty\psi_s(\mathcal S_1)$, the orbit $O^*$ of $H_f$ must be in the region $\Omega_1$. By the uniqueness theorem of solutions of a smooth differential system with respect to the initial point, it follows that $O^*=\psi_{s_0}(\mathcal S_1)$ for some $s_0\in (0,\infty)$. Hence the endpoints $I_1^*$ and $I_2^*$ of $B_1^*$ must coincide with $I_{11}$ and $I_{12}$. This means that $B_1^*$ is a heteroclinic orbit connecting $I_{11}$ and $I_{12}$ when $I_{11}\ne I_{12}$, or a homoclinic one  when $I_{11}= I_{12}$. 
Again by the transversality of $H_f$ and $H_g$ it follows that $B_1^*$ under the action of the flow $\psi_s$ will positively get into the interior of $\Omega_0$. It is in contraction with the definition of $\Omega_0$. Hence, the heteroclinic (or homoclinic) orbit $B_1^*$ for $H_f$ does not exist, and so $\Omega_0$ does not have a boundary in $\mathbb R^2$. Consequently, $\Omega_0$ is empty. Hence $\bigcup\limits_{s=0}\limits^\infty\psi_s(\mathcal S_1)=D_{\mathcal S_1}^+$, which is exactly what we want to prove.
 }

{Now we turn to $\mathcal S_1$ under the action of $\psi_s$ with $s<0$.} Applying the same arguments as above to $\mathcal S_1$ with $s\le 0$ yields that $\bigcup\limits_{s=0}\limits^{-\infty}\psi_s(\mathcal S_1)=D_{\mathcal S_1}^-$. Hence,
$\bigcup\limits_{s=-\infty}\limits^{\infty}\psi_s(\mathcal S_1)$ is the full finite plane $\mathbb R^2$. This proves the claim. \end{proof}

{
\begin{remark}\label{r0}
The last proofs 
verify the next facts.
\begin{itemize}
\item The finite plane $\mathbb R^2$ is foliated by $\{\psi_s(\mathcal S_1)|\ s\in\mathbb R\}$, each of which is a regular orbit of $H_f$.
\item All orbits of $H_f$ except the line at infinity are heteroclinic to the same singularities at the infinity, or homoclinic to the same singularity at the infinity. The same argument applies to $H_g$.
\end{itemize}
\end{remark}
}

\noindent{\bf Claim 6.} {\it  If $D(f,g)\in\mathbb R^*$, then $\mathcal S$ has a unique connected branch.}

In fact, by contrary we suppose that $\mathcal S$ has more than one connected branches. Let $\mathcal S_1$ and $\mathcal S_2$ be two different connected branches of $\mathcal S$, with $p_0\in\mathcal S_1$. By Claim 5 there exists an $s_0\in \mathbb R$ such that $\mathcal S_2=\psi_{s_0}(\mathcal S_1)$. Let $r_0\in\mathcal S_2$ be the image of $p_0$ under the diffeomorphism $\psi_{s_0}$. Then $f(p_0)=f(r_0)$ by the assumption of the contrary. In addition, $g(\psi_s(p_0))\equiv g(p_0)$ for all $s\in\mathbb R$ because $g$ and $\psi_s$ are respectively the first integral and the flow of $H_g$.

Since $f$ is a first integral of the Hamiltonian vector field $H_f$ and $\psi_s(\mathcal S_1)$ is an orbit of $H_f$ for any fixed value of $s$, it follows that $f(\psi_s(\mathcal S_1))$ is a constant for any fixed $s\in\mathbb R$.
Set
\[
\omega(s)=f(\psi_s(\mathcal S_1)),\quad s\in [0,\ s_0].
\]
Then $\omega(s)$ is well defined and is an analytic function, where we have used the facts that
 \begin{itemize}
\item $f$ is a polynomial.
\item $\psi_s$ is an analytic flow, because the Hamiltonian vector field $H_g$ is polynomial, and so is analytic.
\item $\mathcal S_1$ is an analytic curve, because it is a regular orbit of the polynomial vector field $H_f$.
\end{itemize}
By the definition of $\omega(s)$ one has $\omega(0)=\omega(s_0)$. It forces that there exists an $s^*\in (0,\ s_0)$ at which $\omega$ takes a maximum or a minimum value on $[0,s_0]$. Then $\omega'(s^*)=0$. On the other hand,
\[
\begin{split}
\omega'(s^*)&=\left.\frac{df(\psi_s(p_0))}{ds}\right |_{s=s^*}=\left.\nabla f(\psi_s(p_0))\frac{d\psi_s(p_0)}{ds}\right |_{s=s^*}\\
&=\left.\nabla f(\psi_s(p_0)) H_g(\psi_s(p_0))\right |_{s=s^*}=-(f_xg_y-f_yg_x)(\psi_{s^*}(p_0))\\
&=-D(f,g)(\psi_{s^*}(p_0))=-c\ne 0,
\end{split}
\]
a contradiction. Recall that $c\ne 0$ is the real number such that $D(f,g)=c$ on $\mathbb R^2$. This verifies that $\mathcal S$ cannot have two different connected branches. Consequently, the claim follows. \qed

Since $\mathcal S$ has a unique connected branch, the points $p_0,q_0\in\mathbb R^2$ with $p_0\ne q_0$ satisfying $f(p_0)=f(q_0)$ must be located on the same orbit, i.e. $\mathcal S$, of $H_f$.

\noindent{\bf Claim 7.} {\it  If $D(f,g)\in\mathbb R^*$ and $f(p_0)=f(q_0)$, then $g(p_0)\ne g(q_0)$.}

Indeed, by contrary we assume that $g(p_0)= g(q_0)$, and let this common value to be $\nu_0$. Define $\mathcal R:=\{r\in\mathbb R^2|, g(r)=v_0\}$. The similar proof to Claim 6 yields that $\mathcal R$ contains a unique connected branch, which is an orbit of $H_g$. This implies that the orbits $\mathcal S$ of $H_f$ and $\mathcal R$ of $H_g$ intersect at the two different points $p_0$ and $q_0$. But it is impossible, because the flows $\phi_t$ and $\psi_s$ of the vector fields $H_f$ and $H_g$ commute due to $[H_f,H_g]\equiv 0$, and so one of the flows, saying $\psi_s$, maps any orbit of $\phi_t$ to a different one, and vice versa. Another explanation on this impossibility is that there will exist at least one point in between $p_0$ and $q_0$ along $\mathcal S$ at which the Hamiltonian vector fields $H_f$ and $H_g$ are tangent, a contradiction to Claim 3. This proves the claim. \qed

Summarizing Claims 1 to 7 achieves that for any two points $p,q\in\mathbb R^2$ with $p\ne q$, one has $f(p)\ne f(q)$ or $g(p)\ne g(q)$. Consequently, the map $F=(f,g):\ \mathbb R^2\rightarrow \mathbb R^2$ is injective.

It completes the proof of Theorem \ref{tj}.
\end{proof}

As a consequence of the results obtained in the proof of Theorem \ref{tj}, one gets the geometry of the level curves of $f$ and $g$ for the planar real polynomial map $F=(f,g)$.

\begin{corollary}\label{cj}
Assume that $F=(f,g):\ \mathbb R^2\rightarrow \mathbb R^2$ is a polynomial map such that the Jacobian
$DF\in\mathbb R^*$, the following statements hold.
\begin{itemize}
\item[$(a)$] Each of the level curves of $f$ $($and also of $g$$)$ has a unique connected branch with its two endpoints at the infinity, and {the level curves $($which are one-dimensional smooth manifolds$)$ foliate the real plane $\mathbb R^2$.}
\item[$(b)$] Each of the level curves of $g$ intersects all the level curves of $f$ transversally, and vice versa.
\item[$(c)$] Each of the level curves of $f$ and $g$ is closed in $\mathbb RP^2$ with possibly a singularity at the infinity.
\end{itemize}
\end{corollary}

\begin{proof} Statements $(a)$ and $(b)$ follow directly from the results given in the proof of Theorem \ref{tj}. Statement $(c)$ is a consequence of Claim 2 and the facts that a nonsingular variety of a real homogeneous polynomial $h (x,y,z)$  in $\mathbb RP^2$ is a compact one dimensional manifold (see e.g. Milnor \cite{Mi65} and Wilson \cite{Wi78}), and that a polynomial in $\mathbb R^2$ can be projectivized in $\mathbb RP^2$. {Statement $(c)$ can also be obtained partially by identifying the diametrally opposite points on the boundary of the Poincar\'e disc, which produce $\mathbb RP^2$}.

We mention that as shown in the following Remark \ref{r2}, a nonsingular real planar curve in $\mathbb R^2$ could have a singularity at the infinity of $\mathbb RP^2$.
\end{proof}


\begin{example}\label{Ex1}
{Here we provide some other examples illustrating} the geometry of the level curves of $f$ and $g$, which satisfy the Jacobian conjecture in $\mathbb R^2$.
\begin{itemize}
\item[$(E_1)$] {The map $F_1=(f_1,g_1)=(y-(2x-y)^4, \ 2x-y):\mathbb R^2\rightarrow \mathbb R^2$ has its determinant equal to $-2$. The level curves of $f_1$ all have their endpoints at the infinity of the Poincar\'e disc in the direction $y=2 x$ with $x>0$, and are closed in the Poincar\'e compactification with a singularity at the infinity. Whereas the level curves of $g_1$ all have their endpoints at the infinities of the Poincar\'e disc in the two directions of the line $y=2x$, and are all nonsingular and closed in $\mathbb RP^2$. For a definition on singular point of a curve, see for example, \cite[page 31]{Fu08}.

The argument that the nonsingular level curve $f_1=c_1$ in $\mathbb R^2$ is singular at the infinity of $\mathbb RP^2$ follows from the facts that the projectivization of $f_1=c_1$ in the projective coordinate $[X:Y:Z]$ is  $f_1^*(X,Y,Z)=YZ^3-(2X-Y)^4-c_1Z^4=0$ with the gradient of $f_1^*$ vanishing at $[1:2:0]$, and its local expression at the infinity $[1:2:0]$ is $2Z^3-16 X^4- {c_1} Z^4+O(|X,Z|^5)=0$, which has two different branches at $[1:2:0]$ for $Z\ge 0$.

The argument that the level curve $g_1=d_1$ is nonsingular and closed in $\mathbb R^2P$ follows from the facts that the projectivization of $g_1=d_1$ in the projective coordinate $[X:Y:Z]$ is $g_1^*(X,Y,Z)=2X-Y-d_1Z=0$, which is nonsingular in $\mathbb RP^2$.
 }

\item[$(E_2)$] The map $F_2=(f_2,g_2)=(y-x^3, \ y-x-x^3):\mathbb R^2\rightarrow \mathbb R^2$ has its determinant equal to $1$. The level curves of $f_2$ and $g_2$ all have their endpoints at the infinity of the Poincar\'e disc with one in the positive $y$ direction and another in the negative $y$ direction, and they are all closed {and singular} in $\mathbb RP^2$.

\item[$(E_3)$] {The map $F_3=(f_3,g_3)=(y-x^2, \ y-x-x^2):\mathbb R^2\rightarrow \mathbb R^2$ has its determinant equal to $1$. The level curves of $f_3$ and $g_3$ all have their endpoints at the infinity of the Poincar\'e disc in the positive $y$ direction, and they are all closed and nonsingular in $\mathbb RP^2$. Here we have used the fact that any irreducible conic is nonsingular in $\mathbb RP^2$ \cite[page 55]{Fu08}, or that an irreducible conic is isomorphic to $\mathbb RP^1$  \cite[page 78]{Fu08}.}
\end{itemize}
{
These examples show that the level curves of the two polynomials in the polynomial maps $\mathbb R^2\longrightarrow \mathbb R^2$ satisfying the Jacobian conjecture can be either all nonsingular closed ones, or both singular closed ones, or nonsingular closed ones for one polynomial and singular closed one for another polynomial.
}
\end{example}

\begin{remark}\label{r1} {\rm As an application of the results from the proof of Theorem \ref{tj}, one gets that if $f=1+x-x^2y$ (see \cite{Pi94}), then for any polynomial $g(x,y)$, the map $F=(f,g):\ \mathbb R^2\rightarrow \mathbb R^2$ is not injective, because $f=c$ with $c\in\mathbb R$ has more than one distinct connected branches. In addition, one could compute via Mathematica that
a polynomial $g(x,y)$ of higher degree such that $[H_f,H_g]\equiv 0$ is generated by $f$, i.e. $g$ is a polynomial of $f$.
}\end{remark}

\begin{remark}\label{r2} {\rm After completing the proof of Theorem \ref{tj}, one wants to apply our tools to prove the Jacobian conjecture in $\mathbb C^2$. But the complicated topology of planar algebraic curves in $\mathbb C^2$ maybe prevent its direct application, because any polynomial $f(x,y)$ in $\mathbb C^2$ can be projectivized in $\mathbb CP^2$, and as it is well known that for any real homogeneous polynomial $F(x,y,z)$ of degree $m$, if $\mathcal V(F):=\{{[x:y:z]\in}\mathbb CP^2|\ F(x,y,z)=0\}$ is nonsingular in $\mathbb CP^2$, then it is a Riemann surface of genus $g=(m-1)(m-2)/2$. See e.g. Wilson \cite{Wi78}.

{In addition, as shown in Example \ref{Ex1} a nonsingular curve in $\mathbb C^2$ could have singularities at the infinity. This implies that in general, a nonsingular planar curve in $\mathbb C^2$ could have more complicated topology in $\mathbb CP^2$.}

}
\end{remark}

{
\begin{remark}\label{r3} {\rm Our proofs, especially those on Claim $5$, provide also some information of the Hamiltonian vector field $H_f$ (and of $H_g$) at the infinity. In this direction we should mention the next results, which were originally proved for $\mathbb C[x,y]$, but one can check that they also work for real polynomials.


By definition, a polynomial $f\in\mathbb R[x,y]$ has a \textit{Jacobian mate} if there exists a $g\in\mathbb R[x,y]$
such that $D(f,g)=1$.

\noindent {\bf Theorem A.} (Abhyankar \cite{Ab1977,Ab1979}). {\it The following two statements are equivalent.
\begin{itemize}
\item[$(a)$] Every polynomial $f\in\mathbb R[x,y]$ which has a Jacobian mate has a point
at infinity.
\item[$(b)$] The Jacobian conjecture for polynomials in $\mathbb R[x,y]$ holds.
\end{itemize}
}

Theorem A was also proved in \cite[Theorem 10.2.23]{Es00} for complex polynomials
in $\mathbb C[x,y]$ with its proof adapting for real polynomials of $\mathbb R[x,y]$.
By Theorem A it follows that the Jacobian conjecture only holds if the Hamiltonian
vector fields $H_f$ and $H_g$ have only one pair of singularities at the infinity.

Our proofs do not provide exact information on the number of pairs of singularities at the
infinity. But combining Remark \ref{r0} and Claim 5 and its proofs, it follows that all orbits of $H_f$ (resp. $H_g$) in $\mathbb R^2$ are heteroclinic to the same two singularities of $H_f$ (resp. $H_g$) at the infinity, or homoclinic to
the same singularity of $H_f$ (resp. $H_g$) at the infinity. Then applying Theorem A one has that the global topological phase portrait of $H_f$, and also of $H_g$ is one of the two pictures illustrated in Fig.~\ref{F1}.

\begin{figure}[ht]
\begin{center}
\includegraphics[width=10cm]{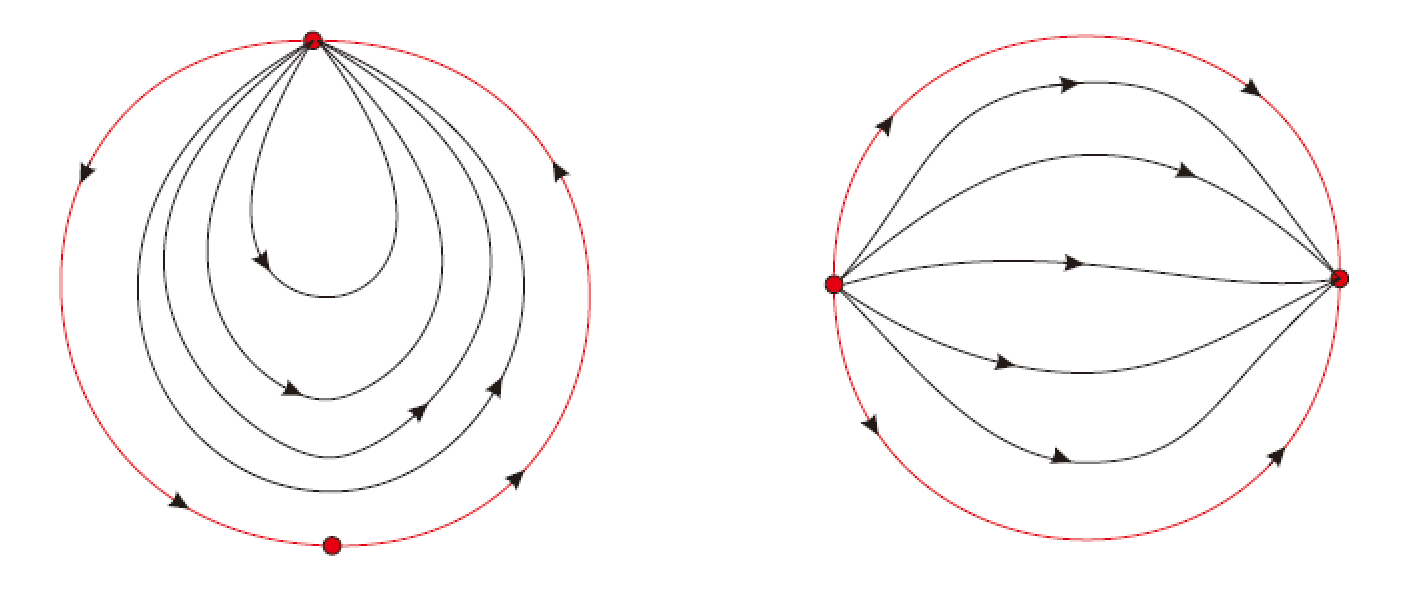}
\end{center}
\caption{Global phase portraits of $H_f$, and of $H_g$} \label{F1}
\end{figure}
}
\end{remark}
}


{
The next result has been used in the proof of Claim 5. It should be known, but we did not find a clear proof on it in some references. So we state it and prove it here for completeness.

\begin{proposition}\label{r5}
Let $f(x,y)$ be a real polynomial in $\mathbb R[x,y]$ of degree $m$, and let $f_m(x,y)$ be the highest order homogeneous part of $f$. Then the Hamiltonian vector field $H_f$ has its singularities at the infinity of the Poincar\'e disc in one to one correspondence with the real linear factors of $f_m(x,y)$.
\end{proposition}
\begin{proof}
By the well known techniques on Poincar\'e compactification in the qualitative theory, see e.g. Cima and Llibre \cite{CL90}, Dumortier {\it et al} \cite{DLA06}, or Ye \cite{Ye86}, the Hamiltonian system
\[
\dot x=-H_y,\quad \dot y=H_x,
\]
in the local charts $[1:y:z]$ on the Poincar\'e disc has the expression
\[
\dot z=z^2f_y\left(\frac{1}{z},\frac yz\right), \quad
z\dot y=z^2f_x\left(\frac{1}{z},\frac yz\right)+yz^2 f_y\left(\frac{1}{z},\frac yz\right),
\]
where $f_x$ and $f_y$ are respectively the partial derivatives of $f(x,y)$ with respect to the first and second variables.
Taking the time rescaling $dt=z^{m-2}d\tau$, this last system can be further written as
\begin{equation}\label{efs}
\frac{dz}{d\tau}=zp(y,z),\quad
\frac{dz}{d\tau}=zq(y,z)+mf_m(1,y),
\end{equation}
where $p$ and $q$ are polynomials in $y$ and $z$, and we have used the fact that
\[
z^2\frac{\partial f_m}{\partial x}\left(\frac{1}{z},\frac yz\right)+yz^2\frac{\partial f_m}{\partial y}\left(\frac{1}{z},\frac yz\right)
=z^3mf_m\left(\frac{1}{z},\frac yz\right)=mz^{3-m}f_m(1,y),
\]
which can also be obtained from the Euler formula on homogenous polynomials, where as above $\frac{\partial f_m}{\partial x}$ and $\frac{\partial f_m}{\partial y}$ are the partial derivatives of $f_m(x,y)$ with respect to $x$ and $y$, respectively.

Equation \eqref{efs} shows exactly that the singularities at the infinity on the Poincar\'e disc in the local chart $[1:y:z]$ are in one to one correspondence with the real zeros of $f_m(1,y)$. Applying the similar arguments to the local chart $[x:1:z]$ on the Poincar\'e disc one gets that the singularities at the infinity on the Poincar\'e disc are in one to one correspondence with the real zeros of $f_m(x,1)$. Note that the real zeros of $f_m(x,1)$ are different from those of $f_m(1,y)$ are only $(0,1)$ and $(0,1)$ (if they exist).
This verifies that the singularities of $H_f$ at the infinity on the Poincar\'e disc are in one to one correspondingly  determined by the real linear factors in the form $ax+by$ of $f_m(x,y)$, whose number is at most $m$ taking into account multiplicity.

It completes the proof of the proposition.
\end{proof}
}

\begin{proposition}\label{r4}
In our setting, i.e. $D(f,g)=c$ a nonzero real number, the flows of the Hamiltonian vector fields $H_f$ and $H_g$ are both complete in $\mathbb R^2$.
\end{proposition}

\begin{proof}
 Indeed, as before let $\phi_t$ and  $\psi_s$ be the flows of the Hamiltonian vector fields $H_f$ and $H_g$, respectively. For any $p\in\mathbb R^2$, and let $I_p=(\alpha_p,\ \beta_p)$ be the maximal open interval of $\mathbb R$ on which the solution $\phi_t(p)$ is defined. Then $0\in I_p$ and it follows from $D(f(x,y),g(x,y))\equiv c$, $(x,y)\in\mathbb R^2$ that
\[
f_x(\phi_t(p))g_y(\phi_t(p))-f_y(\phi_t(p))g_x(\phi_t(p))=c, \quad \mbox{for all \ \ } t\in I_p.
\]
This last equality can be written as
\begin{equation}\label{er5}
\begin{split}
c&=\langle H_f(\phi_t(p)),\ \nabla g(\phi_t(p))\rangle\\
&=\left\langle\frac{d\phi_t(p)}{dt},\ \nabla g(\phi_t(p))\right\rangle
=\frac{dg(\phi_t(p))}{dt}, \qquad  t\in I_p,
\end{split}
\end{equation}
where $\langle\cdot, \cdot\rangle$ represents the inner product of two vectors.
This shows that along any orbit of $H_f$, the value of the polynomial $g$ monotonically increases if $c>0$, and decreases if $c<0$.

By contrary, if there exists a $p_*\in\mathbb R^2$ and the orbit $\phi_t(p_*)$ positively or negatively approaches infinity
in a finite time, without loss of generality we assume that $\beta_{p_*}$ is finite.  Let $\Gamma_*^+$ be the positive orbit of $H_f$ passing $p_*$, and let $\Upsilon_*$ be the orbit of $H_g$ passing $p_*$. {We conclude that
\begin{itemize}
\item $\Omega_*:=\bigcup\limits_{t=0}\limits^{\beta_{p_*}}\phi_t(\Upsilon_*)$ covers one of the half spaces limited by $\Upsilon_*$. 
\end{itemize}
This argument can be obtained in a similar technique as that in the proof of Claim 5. For completeness and readers' convenience we present it here.

Set $L_*:=\lim \limits_{t\rightarrow{\beta_{p_*}^-}}\phi_t(\Upsilon_*)$, a part of the boundary of $\Omega_*$. Note that $p_*\in\Upsilon_*$ and $p_\infty:=\lim \limits_{t\rightarrow{\beta_{p_*}^-}}\phi_t(p_*)\in L_*$ is at the infinity by the  assumption.
By contrary, if $L_*$ is not globally at the infinity, then $L_*$ contains points in the finite plane $\mathbb R^2$. Choose $q_*\in L_*\cap\mathbb R^2$ such that the regular orbit $L_{q_*}$ of $\psi_s$ passing $q_*$, a part of $L_*$, has one of its endpoints at $p_\infty$ (due to the contrary assumption). Then for any $t_0\in(0,\beta_{p_*})$, it follows from 
commutation of the flows $\phi_t$ and $\psi_s$ that $\phi_{-t_0}(L_{q_*})$ is a regular orbit of the flow $\psi_s$. And by Lemma \ref{lep} and its proof one gets that $\phi_{-t_0}(L_{q_*})$ has the same slope at the infinity as that of $L_{q_*}$, where in the application of Lemma \ref{lep} and its proof we have used the fact that a time rescaling does not change the topological structure  and their slopes at the infinity of the orbits of an autonomous differential system.
On the other hand, since $\phi_{-t_0}(q_*)$ is located in the interior of $\Omega_*$, by the uniqueness theorem of the solution of initial value problems we must have $\phi_{-t_0}(L_{q_*})$ coincides with $\phi_{t_*}(\Upsilon_*)$ for some $t_*\in (0,\beta_{p_*})$. But it is not possible, because
in one way $\phi_{t_*}(\Upsilon_*)$ is transversal with $\Gamma_*^+$ at $\phi_{t_*}(p_*)$ due to the transversality of $H_f$ and $H_g$ and commutation of $\phi_t$ and $\psi_s$, and
in another way $\phi_{-t_0}(L_{q_*})$ has the same slope at the infinity as that of $L_{q_*}$. This contradiction implies that
$\Omega_*:=\bigcup\limits_{t=0}\limits^{\beta_{p_*}}\phi_t(\Upsilon_*)$ covers one of the half spaces limited by $\Upsilon_*$.

We now apply this last conclusion to prove that $\beta_{p_*}=\infty$.} Integrating \eqref{er5} for $p$ replaced by $\Upsilon_*$ in $t$ from $0$ to $\beta_{p_*}$ gives
\[
g(\phi_{\beta_{p_*}}(\Upsilon_*))=g(\phi_0(\Upsilon_*))+c\beta_{p_*}=g(\Upsilon_*)+c\beta_{p_*}=g(p_*)+c\beta_{p_*}.
\]
This {together with the contrary assumption $\beta_{p_*}$ being finite} implies that $g$ is bounded in the half space $\Omega_*$. But it is impossible because $g$ is a nonconstant polynomial. This contradiction verifies that $\beta_{p_*}=\infty$. Consequently, the flow of $H_f$ is complete in $\mathbb R^2$.

It completes the proof of the proposition. 
\end{proof}

Finally we illustrate that for Strong Jacobian Conjecture, the flow of the Hamiltonian vector field $H_f$ or $H_g$ could be not complete in $\mathbb R^2$. Take $(f,g)=(-(1+x^2)y,\ x)$. Then $D(f,g)=1+x^2\ge 1$, and the Hamiltonian vector field associated to $f$ is
\[
H_f=\left(\begin{array}{c}
1+x^2\\ -2xy
\end{array}\right).
\]
It has the flow
\[
\phi_t(c_1,c_2)=\left(\tan(t+\arctan c_1)),\ (1+c_1^2){c_2}\cos(t+\arctan c_1)^2\right)
\]
satisfying $\phi_0(c_1,c_2)=(c_1,c_2)$, defined on $(\arctan c_1-\pi/2, \ \arctan c_1+\pi/2)$ for arbitrary  $(c_1,c_2)\in\mathbb R^2$.

\section*{Acknowledgments}

I sincerely appreciate Professor Jaume Llibre from Universitat Aut\`{o}noma de Barcelona SPAIN, who has read the first version of this paper, and provided several valuable suggestions and comments (for instance, subtle on the proof of Claim 5 and the information on Theorem $A$). { The author also thanks professors
Changjian Liu from Sun Yat-sen University CHINA,
David Marin from Universitat Autonoma de Barcelona SPAIN,
Cleto B. Miranda-Neto from Universidade Federal da Para\'{\i}ba BRAZIL,
Marco Sabatini from Universit\`{a} di Trento ITALY,
and
Wilberd van der Kallen from Universiteit Utrecht The NETHERLANDS,
who read the version arXiv:2011.12701, and provided some nice suggestions and comments.
Some special gratitude should be given to professor Jiangong You, who invited the author to talk this result in Chern Institute of Mathematics, Nankai University CHINA, where he and professor Yiming Long presented some good comments.
}

The author is partially supported by NNSF of China grant numbers 11671254, 11871334 and 12071284{, and by Innovation Program of Shanghai Municipal Education Commission grant number 2021-01-07-00-02-E00087}.


\end{document}